\documentclass{article}
\usepackage{tikz}
\usetikzlibrary{arrows.meta,bending,calc,intersections,cd}

\usepackage{amsmath}
\usepackage{amssymb}
\usepackage{fullpage}
\usepackage[all]{xy}
\usepackage{amsthm}
\usepackage{graphicx}
\usepackage{subfig}
\usepackage{enumerate}
\usepackage{tikz-cd}
\usepackage{bookmark}
\allowdisplaybreaks

\title{Gendo-Frobenius algebras and comultiplication}
\author{\c{C}i\u{g}dem Y{\i}rt{\i}c{\i} \footnote{Institute of Algebra and Number Theory, University of Stuttgart, Pfaffenwaldring 57, 70569 Stuttgart, Germany,\hspace{3cm}cigdemyirtici@gmail.com.}
}
\date{}

\theoremstyle{plain}
\newtheorem{theorem}{Theorem}[section]

\newtheorem{lemma}[theorem]{Lemma}
\newtheorem{corollary}[theorem]{Corollary}
\newtheorem{proposition}[theorem]{Proposition}
\theoremstyle{definition}
\newtheorem{definition}[theorem]{Definition}
\newtheorem{example}[theorem]{Example}
\theoremstyle{remark}
\newtheorem{remark}[theorem]{Remark}

\newtheorem*{theorem*}{Theorem}

\begin{document}
\maketitle

\begin{abstract}
Gendo-Frobenius algebras are a common generalisation of Frobenius algebras and of gendo-symmetric algebras. A comultiplication is constructed for gendo-Frobenius algebras, which specialises to the known comultiplications on Frobenius and on gendo-symmetric algebras. In addition, Frobenius algebras are shown to be precisely those gendo-Frobenius algebras that have a counit compatible with this comultiplication. Moreover, a new characterisation of gendo-Frobenius algebras is given. This new characterisation is a key for constructing the comultiplication of gendo-Frobenius algebras.\\

\end{abstract}
\section{Introduction}
Group algebras of finite groups have two different comultiplications, one as a Hopf algebra (see \cite{Yamagata}, Chapter VI) and another one as a symmetric algebra. The second comultiplication can be extended to Frobenius algebras, where it plays an important role in relating commutative Frobenius algebras with two-dimensional topological quantum field theories \cite{Kock}. Another generalisation of the second comultiplication can be obtained for gendo-symmetric algebras \cite{Koenig1}, which include the algebras on both sides of classical Schur-Weyl duality and of Soergel’s structure theorem for the BGG-category $\mathcal{O}$, and many other algebras of interest. The aim of this article is to extend this second comultiplication to gendo-Frobenius algebras, which include both Frobenius algebras and gendo-symmetric algebras.

Motivated by \cite{Koenig1} and \cite{Kerner}, we call a finite dimensional $k$-algebra $A$ a {\em gendo-Frobenius algebra} if it satisfies one of the following equivalent conditions:

(i) $A$ is isomorphic to the endomorphism algebra of a finite dimensional faithful right module $M$ over a Frobenius algebra $B$ such that $M\cong M_{\nu_B} $ as right $B$-modules, where $\nu_B$ is a Nakayama automorphism of $B$.

(ii) Hom$_A(\text{D}(A), A)\cong A$ as left $A$-modules.

The equivalence of the conditions (i) and (ii) has been proved by Kerner and Yamagata in \cite{Kerner}. Gendo-Frobenius algebras are Morita algebras, that is, they are isomorphic to endomorphism algebras of finite dimensional faithful modules over self-injective algebras. In the definition of Morita algebras, the condition (ii) given above is relaxed, and requires that $\text{Hom}_A({}_A \text{D}(A),{}_A A)$ is a faithful left $A$-module \cite{Kerner}. Therefore, Morita algebras are not always gendo-Frobenius. We may visualise the hierarchy of the finite dimensional algebras mentioned above as follows.

$$
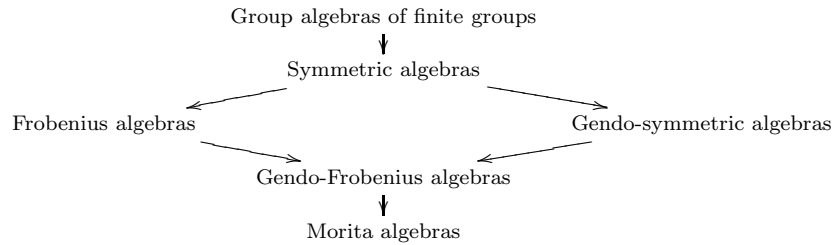
\begin{xy}
\xymatrixrowsep{0.1in}
\xymatrixcolsep{0.1in}
  \xymatrix{
   & &\text{\footnotesize{Group algebras of finite groups}} \ar[d] & \\
   & & \ar[dl] \text{\footnotesize{Symmetric algebras}} \ar[dr] & \\
   &\text{\footnotesize{Frobenius algebras}}\ar[dr] & & \ar[dl] \text{\footnotesize{Gendo-symmetric algebras}} \\
   & &\text{\footnotesize{Gendo-Frobenius algebras}} \ar[d] &\\
   & &\text{\footnotesize{Morita algebras}} & }
\end{xy}$$
In the above diagram, an arrow means the class on top is contained by the class below.\\

In this article, we give a new characterisation of gendo-Frobenius algebras. Moreover, we construct a comultiplication for gendo-Frobenius algebras, which specialises to the known comultiplications on Frobenius and on gendo-symmetric algebras. In addition, we show that Frobenius algebras are precisely those gendo-Frobenius algebras that have a counit compatible with this comultiplication. The new characterisation of gendo-Frobenius algebras is a key for constructing the comultiplication for gendo-Frobenius algebras.\\

\textbf{Main results.} (a) (\emph{Theorem} \ref{PropgFsigma}) \emph{Let} $A$ \emph{be a finite dimensional} $k$\emph{-algebra. Then} $A$ \emph{is gendo-Frobenius if and only if there exists an automorphism} $\sigma \in \text{Aut}(A)$ \emph{such that} $\text{D}(A)_{\sigma^{-1}}\otimes_A \text{D}(A)\cong \text{D}(A)$ \emph{as} $A$\emph{-bimodules and} $\sigma$ \emph{is uniquely determined up to an inner automorphism.}

(b) (\emph{Theorem} \ref{thmcomgenFrob} \& \emph{Proposition} \ref{propgenFrobFrobcounit}) \emph{Let} $A$ \emph{be a gendo-Frobenius algebra. Then there is a coassociative comultiplication} $\Delta: A \rightarrow  A \otimes_k A$ \emph{which is an} $A$\emph{-bimodule morphism}. \emph{In addition,} $(A,\Delta)$ \emph{has a counit if and only if} $A$ \emph{is Frobenius.}

\section{Preliminaries}
In this section, we give some necessary definitions, notions and results for introducing gendo-Frobenius algebras and their comultiplication. Throughout, all algebras and modules are finite dimensional over an arbitrary field $k$ unless stated otherwise. By $\text{D}$, we denote the usual $k$-duality functor $\text{Hom}_k(-,k)$.

Let $A$ be a finite dimensional $k$-algebra and $\omega$ be an automorphism of $A$. Suppose that $M$ is a left $A$-module. Here, ${}_{\omega} M$ is the left $A$-module such that ${}_{\omega} M = M$ as $k$-vector spaces and the left $A$-module structure is defined by $a \cdot m = \omega(a)m$ for all $a \in A$ and $m \in M$. Similarly, for a right $A$-module N, $N_{\omega}$ is the right $A$-module such that $N_{\omega} = N$ as $k$-vector spaces and the right $A$-module structure is defined by $n \cdot a = n\omega(a)$ for all $a \in A$ and $n \in N$. The automorphism group of an algebra $A$ is denoted by Aut$(A)$.

\begin{definition}\label{defnFrob} A finite dimensional $k$-algebra $A$ is called {\em Frobenius} if it satisfies one of the following equivalent conditions:

(i) There exists a linear form $\varepsilon : A \rightarrow k$ whose kernel does not contain a nonzero left ideal of $A$.

(ii) There exists an isomorphism $\lambda_L: A \rightarrow \text{D}(A)$ of left $A$-modules.

(iii) There exists a linear form $\varepsilon' : A \rightarrow k$ whose kernel does not contain a nonzero right ideal of $A$.

(iv) There exists an isomorphism $\lambda_R : A \rightarrow \text{D}(A)$ of right $A$-modules.
\end{definition}
This definition is based on \cite{Yamagata}, Theorem IV.2.1, which provides the equivalence of the four conditions.\\

The linear form $\varepsilon: A \rightarrow k$ in Definition \ref{defnFrob} is called {\em Frobenius form} and it is equal to $\lambda_L(1_A)$.
\begin{definition}\label{defnNakAut} An automorphism $\nu$ of a Frobenius algebra $A$ is called a {\em Nakayama automorphism} if $A_{\nu} \cong \text{D}(A)$ as $A$-bimodules.
\end{definition}
Every Frobenius algebra $A$ has a Nakayama automorphism which is unique up to inner automorphisms (\cite{Yamagata}, Corollary IV.3.5). We denote by $\nu_A$ a Nakayama automorphism of $A$.\\

We now give the definition of symmetric algebras which are special Frobenius algebras.
\begin{definition}\label{defnSym} A finite dimensional $k$-algebra $A$ is called {\em symmetric} if it satisfies one of the following equivalent conditions:

(i) There exists a linear form $\varepsilon : A \rightarrow k$ such that $\varepsilon(ab)=\varepsilon(ba)$ for all $a,b\in A$, and whose kernel does not contain a nonzero one-sided ideal of $A$.

(ii) There exists an isomorphism $\lambda: A \rightarrow \text{D}(A)$ of $A$-bimodules.
\end{definition}
This definition is based on \cite{Yamagata}, Theorem IV.2.2, which provides the equivalence of the two conditions.\\

A Frobenius algebra $A$ is symmetric if and only if $\nu_A$ is inner (\cite{Yamagata2}, Theorem 2.4.1). In this case, we may take the identity automorphism as a Nakayama automorphism.\\

In \cite{Abrams}, Abrams proved that Frobenius algebras are characterised by the existence of a comultiplication with properties like counit and coassociative:
\begin{theorem} \label{thm1}(\cite{Abrams}, Theorem 2.1) An algebra $A$ is a Frobenius algebra if and only if it has a coassociative counital comultiplication $\alpha: A \rightarrow A \otimes_k A$ which is a map of $A$-bimodules.
\end{theorem}
Let $A$ be a Frobenius algebra and $\mu: A\otimes_k A \rightarrow A$ be the multiplication map. Since $A$ is Frobenius, there is a left $A$-module isomorphism $\lambda_L : A \cong \text{D}(A)$. Then we obtain a comultiplication $\alpha_L: A \rightarrow A \otimes_k A$ which is the composition $(\lambda_L ^{-1} \otimes_k \lambda_L^{-1})\circ \mu^* \circ \lambda_L$. Similary, we can define $\alpha_R$ which is the composition $(\lambda_R ^{-1} \otimes_k \lambda_R^{-1})\circ \mu^* \circ \lambda_R$. Abrams proved that $\alpha_L=\alpha_R$. Therefore, $\alpha:= \alpha_L=\alpha_R$. Here, $\varepsilon = \lambda_L(1_A)$ serves as a counit for $\alpha$ and this $\varepsilon$ is actually the Frobenius form of $A$.\\

To give the definition of gendo-symmetric algebras, we need the following concept.\\

\emph{Dominant dimension.} Let $A$ be a finite dimensional $k$-algebra. The dominant dimension of $A$ is at least $d$ (written as domdim$(A) \geq d$) if there is an injective coresolution
\[0 \longrightarrow A \longrightarrow I_0 \longrightarrow I_1 \longrightarrow \cdot\cdot\cdot \longrightarrow I_{d-1} \longrightarrow I_d \longrightarrow \cdot\cdot\cdot \]
such that all modules $I_i$ where $0\leq i \leq d-1$ are also projective.\\

A finite dimensional left $A$-module $M$ is said to have \emph{double centraliser property} if the canonical homomorphism of algebras $f: A \rightarrow \text{End}_B (M)$ is an isomorphism for $B= \text{End}_A(M)^{op}$.

If domdim$(A) \geq 1$, then $I_0$ in the definition of dominant dimension is projective-injective and up to isomorphism it is the unique minimal faithful left $A$-module. Therefore, it is of the form $Ae$ for some idempotent $e$ in $A$. Note that $Ae$ is a generator-cogenerator as a right $eAe$-module. If further domdim$(A) \geq 2$, then $Ae$ has double centraliser property, namely, $A \cong \text{End}_{eAe}(Ae)$ canonically.

\begin{definition}\label{defnGenSym} A finite dimensional $k$-algebra $A$ is called {\em gendo-symmetric} if it satisfies one of the following equivalent conditions:

(i) $A$ is the endomorphism algebra of a generator over a symmetric algebra.

(ii) $\text{Hom}_A({}_A \text{D}(A), {}_A A)\cong A$ as $A$-bimodules.

(iii) $\text{D}(A)\otimes _A \text{D}(A) \cong \text{D}(A)$ as $A$-bimodules.

(iv) domdim$(A) \geq 2$ and $\text{D}(Ae)\cong eA$ as $(eAe,A)$-bimodules, where $Ae$ is a basic faithful projective-injective $A$-module.
\end{definition}
This definition is based on \cite{Koenig2}, Theorem 3.2, which provides the equivalence of the four
conditions.\\

By condition (iv) in Definition \ref{defnGenSym}, symmetric algebras are gendo-symmetric by choosing $e=1_A$.\\

Gendo-symmetric algebras have a comultiplication with some special properties. In fact, Fang and Koenig gave the following theorem.

\begin{theorem} (\cite{Koenig1}, Theorem 2.4 and Proposition 2.8) Let $A$ be a gendo-symmetric algebra. Then $A$ has a coassociative comultiplication $\Delta : A \rightarrow A\otimes_k A$ which is an A-bimodule morphism. In addition, $(A,\Delta)$ has a counit if and only if $A$ is symmetric.
\end{theorem}

Kerner and Yamagata \cite{Kerner} investigated two generalisations of gendo-symmetric algebras. The most general one is motivated by Morita \cite{Morita} and they called a finite dimensional algebra $A$ a {\em Morita algebra}, if $A$ is isomorphic to the endomorphism algebra of a finite dimensional faithful module over a self-injective algebra. Morita algebras contain both gendo-symmetric and Frobenius algebras.

The second one is defined by relaxing the condition on the bimodule isomorphism in Definition \ref{defnGenSym} (ii), and we focus on this generalisation in this article.

The following lemma and the proof of this lemma are rearranged versions of Lemma 2.4 in \cite{Kerner} and its proof by taking into account the definition of Nakayama automorphism which is used here differently compared to \cite{Kerner}.

\begin{lemma}\label{lemmagenFrob} Let $A$ be a finite dimensional $k$-algebra and $\text{D}(Ae)\cong eA$ as right $A$-modules for an idempotent $e$ of $A$. Then $eAe$ is Frobenius and $ {}_{\nu_{eAe}^{-1}}eA \cong \text{D}(Ae)$ as $(eAe,A)$-bimodules, where $\nu_{eAe}$ is a Nakayama automorphism of $eAe$.
\end{lemma}
\begin{proof} Observe that $(eAe,A)$-bimodules $\text{D}(Ae)$ and $eA$ are faithful $eAe$-modules and $l_{eA} : eAe \rightarrow \text{End}_A (eA)$ and $l_{\text{D}(Ae)} : eAe \rightarrow \text{End}_A (\text{D}(Ae))^{op}$ are isomorphisms. Let us apply Lemma 1.1 in \cite{Kerner} to the right $A$-module isomorphism $eA \cong \text{D}(Ae)$. Then we obtain an $(eAe,A)$-bimodule isomorphism ${}_{\alpha}eA \cong \text{D}(Ae)$ where $\alpha$ is an automorphism of $eAe$. Multiplying $e$ on the right implies an $(eAe,eAe)$-bimodule isomorphism ${}_{\alpha}eAe\cong \text{D}(eAe)$. By taking the dual of this isomorphism, we obtain that $\text{D}(eAe)_{\alpha}\cong eAe$ as $(eAe,eAe)$-bimodules. Therefore, $eAe$ is a Frobenius algebra and $\alpha^{-1}$ is a Nakayama automorphism of $eAe$. Thus, we get ${}_{\nu_{eAe}^{-1}}eA \cong \text{D}(Ae)$ as $(eAe,A)$-bimodules.
\end{proof}
\begin{definition} Let $A$ be a finite dimensional $k$-algebra. An idempotent $e$ of $A$ is called \emph{self-dual} if $\text{D}(eA) \cong Ae$ as left $A$-modules, and \emph{faithful} if both $Ae$ and $eA$ are faithful $A$-modules.
\end{definition}
Observe that self-duality of an idempotent is left–right symmetric. Moreover, an algebra $A$ is a Frobenius algebra if and only if the identity $1_A$ of $A$ is a self-dual idempotent.\\

A hierarchy of the finite dimensional algebras mentioned in this article can be given as follows.

$$
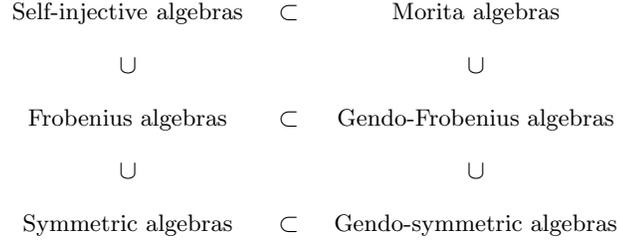
\begin{xy}
\xymatrixrowsep{0.1in}
\xymatrixcolsep{0.1in}
  \xymatrix{
  & \text{\small{Self-injective algebras}} & \subset & \text{\small{Morita algebras}} \\
  & \cup & & \cup\\
  & \text{\small{Frobenius algebras}}  & \subset & \text{\small{Gendo-Frobenius algebras}} \\
  & \cup & & \cup \\
  & \text{\small{Symmetric algebras}} & \subset & \text{\small{Gendo-symmetric algebras}}
   }
\end{xy}$$

Comultiplications of symmetric algebras and Frobenius algebras given on the left part of the above diagram are known by \cite{Abrams}, and comultiplication of gendo-symmetric algebras given on the right part is known by \cite{Koenig1}. Gendo-Frobenius algebras are a common generalisation of Frobenius algebras and of gendo-symmetric algebras. The aim of this article is to construct a comultiplication for gendo-Frobenius algebras, which specialises to the known comultiplications on Frobenius and on gendo-symmetric algebras.

\section{Gendo-Frobenius algebras}
Inspired by \cite{Koenig2}, Kerner and Yamagata considered the case, when the module $\text{Hom}_A (D(A), A)$ is isomorphic to $A$, at least as a one-sided module and they obtained Theorem 3 in \cite{Kerner} which we use as definition of gendo-Frobenius algebras as follows.

\begin{definition}\label{theoremGenFrob} A finite dimensional $k$-algebra $A$ is called {\em gendo-Frobenius} if it satisfies one of the following equivalent conditions:

(i) Hom$_A(\text{D}(A), A)\cong A$ as left $A$-modules.

(ii) Hom$_A(\text{D}(A), A)\cong A$ as right $A$-modules.

(iii) $A$ is a Morita algebra with an associated idempotent $e$ such that $eAe$ is a Frobenius algebra with Nakayama automorphism $\nu_{eAe}$ and $Ae\cong Ae_{\nu_{eAe}} $ as right $eAe$-modules.

(iv) $A$ is a Morita algebra with an associated idempotent $e$ such that $eAe$ is a Frobenius algebra with Nakayama automorphism $\nu_{eAe}$ and $eA\cong {}_{\nu_{eAe}} eA $ as left $eAe$-modules.

(v) $A$ is isomorphic to the endomorphism algebra of a finite dimensional faithful right module $M$ over a Frobenius algebra $B$ such that $M\cong M_{\nu_B} $ as right $B$-modules.

(vi) $A$ is isomorphic to the opposite endomorphism algebra of a finite dimensional faithful left module $N$ over a Frobenius algebra $B$ such that $N\cong {}_{\nu_B}N $ as left $B$-modules.
\end{definition}
\begin{remark} The idempotent $e$ of $A$ in Definition \ref{theoremGenFrob} is self-dual and faithful. See the proof of Theorem 3 in \cite{Kerner}.
\end{remark}
By the conditions (iii) and (iv) in Definition \ref{theoremGenFrob}, Frobenius algebras are gendo-Frobenius by choosing $e=1_A$.

\begin{remark} Kerner and Yamagata \cite{Kerner} proved that a finite dimensional $k$-algebra $A$ is a Morita algebra if and only if $\text{Hom}_A({}_A \text{D}(A),{}_A A)$ is a faithful left $A$-module and $\text{domdim}(A) \geq 2$. Therefore, Morita algebras do, in general, not satisfy the condition (i) and (ii) given in the definition of gendo-Frobenius algebras.
\end{remark}
\begin{example}\label{GenFrobEx2} Let $B$ be the path algebra of the following quiver
$$\begin{xy}
  \xymatrix{
1 \ar@<1ex>[r]^{\beta_1}  &2	\ar@<1ex>[l]^{\beta_2}   }
\end{xy}
$$
such that $\beta_1\beta_2=0=\beta_2\beta_1$. Then $B$ is a nonsymmetric Frobenius algebra and it has a Nakayama automorphism $\nu_B$ such that $\nu_B(e_1)=e_2$, $\nu_B (e_2)=e_1$, $\nu_B(\beta_1)= \beta_2$ and $\nu_B (\beta_2)=\beta_1$. Let $M=B \oplus S_1 \oplus S_2$, where $S_1$ and $S_2$ are simple modules corresponding to $e_1$ and $e_2$, respectively; and $A= \text{End}_B (M)$. Then $A$ is isomorphic to the path algebra of the following quiver
$$\begin{xy}
\xymatrixrowsep{0.2in}
\xymatrixcolsep{0.2in}
  \xymatrix{
    &         & 1 \ar[dl]_{\alpha_1} 	&	 \\
    & 3  \ar[dr]_{\alpha_3}   &	& 4 \ar[ul]_{\alpha_4}	\\
    &         & 2 \ar[ur]_{\alpha_2}          &}
\end{xy}$$
such that $\alpha_3\alpha_2=0=\alpha_4\alpha_1$.

The right $B$-module $M$ is faithful and $M_{\nu_B}\cong M$ as right $B$-modules. Hence, by Definition \ref{theoremGenFrob}, we obtain that $A$ is a gendo-Frobenius algebra.
\end{example}
\begin{remark} Let us consider the algebra $B$ in Example \ref{GenFrobEx2}. Let $M = B \oplus S_1$. Then $M_B$ is faithful and $A = \text{End}_B (M)$ is a Morita algebra. However, $M_{\nu_B} \ncong M$ as right $B$-modules. Hence, by Definition \ref{theoremGenFrob}, $A$ is not gendo-Frobenius.
\end{remark}
\begin{remark} The class of gendo-Frobenius algebras is not closed under Morita equivalences since the property $\text{Hom}_A (\text{D}(A),A)\cong A$ as left (or right) $A$-modules is not Morita invariant.
\end{remark}
The following proposition and the proof of this proposition are the rearranged versions of Proposition 3.5 in \cite{Kerner} and its proof similarly as Lemma \ref{lemmagenFrob}, and it shows that, in case $A$ is gendo-Frobenius, a Nakayama automorphism $\nu_{eAe}$ for a faithful and self-dual basic idempotent $e$ of $A$ extends to an automorphism of $A$.
\begin{proposition}\label{propsigma} Let $A$ be a gendo-Frobenius algebra with a faithful and self-dual idempotent $e$. Then there is an automorphism $\sigma \in \text{Aut}(A)$ such that

(i) Hom$_A(\text{D}(A), A)_{\sigma}\cong A$ as $(A, A)$-bimodules and $\sigma$ is uniquely determined up to an inner automorphism.

(ii) $eA \cong {}_{\nu_{eAe}}eA _{\sigma}$ as $(eAe, A)$-bimodules.

(iii) Moreover, in case $e$ is basic, we can choose the $\sigma$ such that $\sigma(e)=e$ and the restriction of $\sigma$ to $eAe$ is a Nakayama automorphism of $eAe$.
\end{proposition}
\begin{proof} (i) The proof is similar to the proof of Proposition 3.5 (i) in \cite{Kerner}. But here we apply Lemma 1.1 in \cite{Kerner} to the isomorphism ${}_A A \cong {}_A\text{Hom}_A (\text{D}(A),A)$. So we obtain that there is an automorphism $\sigma$ such that $A \cong \text{Hom}_A (\text{D}(A),A)_{\sigma}$ as $(A,A)$-bimodules.

(ii) By applying $e$ on the left side of the $(A,A)$-bimodule isomorphism $A \cong \text{Hom}_A(\text{D}(A), A)_{\sigma}$, we obtain the following $(eAe,A)$-bimodule isomorphisms
\begin{align*}
  eA & \cong e\text{Hom}_A(\text{D}(A), A)_{\sigma} = \text{Hom}_A(\text{D}(A)e, A)_{\sigma} \\
   & = \text{Hom}_A(\text{D}(eA), A)_{\sigma} \cong \text{Hom}_A(Ae_{\nu_{eAe}}, A)_{\sigma}\\
   & = {}_{\nu_{eAe}}\text{Hom}_A(Ae, A)_{\sigma} \cong {}_{\nu_{eAe}}eA_{\sigma}
\end{align*}
since $\text{D}(eA)\cong Ae_{\nu_{eAe}}$ as $(A,eAe)$-bimodules.

(iii) We first replace $\sigma$ in the proof of Proposition 3.5 (iii) in \cite{Kerner} with $\sigma^{-1}$. Then by using the same proof, we obtain that there is a $\theta \in $Aut$(A)$ with $\theta(x) = cxc^{-1}$ for all $x \in A$, where $c$ is an invertible element in $A$ such that $(\theta\sigma^{-1})(e) = e$ and $\theta\sigma^{-1} \in$Aut$(A)$. Observe that Hom$_A(\text{D}(A), A)\cong A_{\sigma^{-1}} \cong A_{\theta\sigma^{-1}}$ as $(A, A)$-bimodules, because $A\cong A_{\theta}$ as $(A, A)$-bimodules. By replacing $\sigma^{-1}$ with $\theta\sigma^{-1}$, we obtain that $\sigma^{-1}(e)=e$, that is, $\sigma(e)=e$. Now, we multiply $e$ on right side of the isomorphism $eA_{\sigma}\cong {}_{\nu^{-1}_{eAe}}eA$ given in (ii). Then we obtain $(eAe,eAe)$-bimodule isomorphisms $eAe_{\sigma_e}\cong {}_{\nu^{-1}_{eAe}} eAe \cong eAe_{\nu_{eAe}}$ , where $\sigma_e$ denotes the restriction of $\sigma$ to $eAe$. By using Lemma II.7.15 and Corollary IV.3.5 in \cite{Yamagata}, we obtain that $\sigma_e = \theta_e \nu_{eAe}$ for some inner automorphism $\theta_e$ of the algebra $eAe$, which shows that $\sigma_e$ is a Nakayama automorphism of $eAe$.
\end{proof}

\section{Comultiplication}
In this section, inspired by \cite{Koenig1}, we construct a coassociative comultiplication (possibly without a counit) for gendo-Frobenius algebras and give its properties.
\begin{lemma}\label{lemmagamma} Let $A$ be a gendo-Frobenius algebra with a faithful and self-dual idempotent $e$. Then there is an automorphism $\sigma \in \text{Aut}(A)$ such that $Ae \otimes_{eAe} eA_{\sigma} \cong D(A)$ as $A$-bimodules and $\sigma$ is uniquely determined up to an inner automorphism.
\end{lemma}
\begin{proof} Following Lemma \ref{lemmagenFrob} and Proposition \ref{propsigma} (ii), fix an $(eAe,A)$-bimodule isomorphism $\tau : eA_{\sigma} \cong \text{D}(Ae)$, where $\sigma \in \text{Aut}(A)$ and it is uniquely determined up to an inner automorphism.

By the double centralizer property of $Ae$ and the isomorphism $\tau$, we obtain the following $A$-bimodule isomorphism
\begin{align*}
A \cong \text{Hom}_{eAe}(Ae, Ae) & \cong \text{Hom}_{eAe}(\text{D}(Ae), \text{D}(Ae))\\
& \cong \text{Hom}_{eAe}(eA_{\sigma}, \text{D}(Ae))\\
& \cong \text{Hom}_k (Ae \otimes_{eAe}  eA_{\sigma} , k ).
\end{align*}
Then by dualising $\text{Hom}_k (Ae \otimes_{eAe} eA_{\sigma}, k ) \cong A$, we obtain that there is an $A$-bimodule isomorphism $\gamma: Ae \otimes_{eAe} eA_{\sigma} \cong D(A)$ such that $\gamma(ae \otimes_{eAe} eb)(x)=\tau (eb\sigma(x))(ae)$ for all $a,b,x \in A$.
\end{proof}
\begin{theorem}\label{PropgFsigma} Let $A$ be a finite dimensional $k$-algebra. Then $A$ is gendo-Frobenius if and only if there exists an automorphism $\sigma \in \text{Aut}(A)$ such that $\text{D}(A)_{\sigma^{-1}}\otimes_A \text{D}(A)\cong \text{D}(A)$ as $A$-bimodules and $\sigma$ is uniquely determined up to an inner automorphism.
\end{theorem}
\begin{proof} Let $A$ be gendo-Frobenius. By the isomorphism $\gamma$, observe that there is an $A$-bimodule isomorphism $\gamma': Ae \otimes_{eAe} eA \cong D(A)_{\sigma^{-1}}$ such that $\sigma \in \text{Aut}(A)$ and it is uniquely determined up to an inner automorphism. Hence, there is an $A$-bimodule isomorphism
\begin{align*}
\epsilon: \text{D}(A)_{\sigma^{-1}}\otimes_A \text{D}(A) & \stackrel{(1)}{\cong} (Ae \otimes_{eAe} eA)\otimes_A (Ae \otimes_{eAe} eA_{\sigma})\\
& \cong Ae \otimes_{eAe} eAe \otimes_{eAe} eA_{\sigma} \\
& \cong Ae \otimes_{eAe} eA_{\sigma} \\
& \cong \text{D}(A),
\end{align*}
where (1) is $\gamma'^{-1}\otimes_A \gamma^{-1}$, and it is explicitly defined by
\begin{align*}
\epsilon: \gamma'(ae\otimes_{eAe} eb) \otimes_A \gamma(ce \otimes_{eAe} ed) & \mapsto (ae \otimes_{eAe} eb)\otimes_A (ce \otimes_{eAe} ed) \\
& \mapsto ae \otimes_{eAe} ebce \otimes_{eAe} ed\\
& \mapsto aebce \otimes_{eAe} ed \\
& \mapsto \gamma(aebce \otimes_{eAe} ed),
\end{align*}
for any $a,b,c,d \in A$.

Now let $\text{D}(A)_{\sigma^{-1}}\otimes_A \text{D}(A)\cong \text{D}(A)$ as $A$-bimodules. Taking the dual of this isomorphism gives the $A$-bimodule isomorphism $\text{Hom}_A ({}_\sigma \text{D}(A),A)\cong A$. Then we obtain the following isomorphisms of $A$-bimodules
\[A \cong \text{Hom}_A ({}_\sigma \text{D}(A),A) \cong \text{Hom}_A (\text{D}(A),A)_\sigma.\]
It means that there is a left $A$-module isomorphism $\text{Hom}_A (\text{D}(A),A)\cong A$ and by Definition \ref{theoremGenFrob}, $A$ is gendo-Frobenius.
\end{proof}
Let $m_1$ be the composition of the canonical $A$-bimodule morphism
\[\phi: \text{D}(A)_{\sigma^{-1}}\otimes_k \text{D}(A) \rightarrow \text{D}(A)_{\sigma^{-1}}\otimes_A \text{D}(A)\]
with the isomorphism $\epsilon$ given in the proof of Theorem \ref{PropgFsigma} such that
\[m_1: \text{D}(A)_{\sigma^{-1}}\otimes_k \text{D}(A) \stackrel{\phi}{\rightarrow}  \text{D}(A)_{\sigma^{-1}}\otimes_A \text{D}(A)\stackrel{\epsilon}{\cong} \text{D}(A),\]
where
\begin{align*}
m_1 : \gamma'(ae \otimes_{eAe} eb)\otimes_k \gamma(ce \otimes_{eAe} ed) &\mapsto \gamma'(ae \otimes_{eAe} eb)\otimes_A \gamma(ce \otimes_{eAe} ed)\\
&\mapsto \gamma (aebce \otimes_{eAe} ed).
\end{align*}
Let $m_2: \text{D}(A) \otimes_k \text{D}(A) \rightarrow  \text{D}(A)_{\sigma^{-1}}\otimes_k \text{D}(A)$ be the map which is defined by
\[m_2(\gamma(ae\otimes_{eAe} eb)\otimes_k \gamma(ce\otimes_{eAe} ed))= \gamma'(ae\otimes_{eAe} eb)\otimes_k \gamma(ce\otimes_{eAe} ed),\]
where $\gamma(ae\otimes_{eAe} eb), \gamma(ce\otimes_{eAe} ed) \in \text{D}(A)$ and $\gamma'(ae\otimes_{eAe} eb) \in \text{D}(A)_{\sigma^{-1}}$.\\

{\bf Claim.} The map $m_2$ is an $A$-bimodule morphism.\\

\emph{Proof of Claim.} It is enough to check that
\[m_2(x\gamma(ae\otimes_{eAe} eb)\otimes_k \gamma(ce \otimes_{eAe} ed))=xm_2(\gamma(ae\otimes_{eAe} eb)\otimes_k \gamma(ce \otimes_{eAe} ed))\]
and
\[m_2(\gamma(ae\otimes_{eAe} eb)\otimes_k \gamma(ce \otimes_{eAe} ed)y)= m_2(\gamma(ae\otimes_{eAe} eb)\otimes_k \gamma(ce \otimes_{eAe} ed))y\]
for any $x,y \in A$. We observe that
\begin{align*}
  m_2(x \gamma(ae\otimes_{eAe} eb) \otimes_k \gamma(ce \otimes_{eAe} ed))& =  m_2(\gamma(xae\otimes_{eAe} eb) \otimes_k \gamma(ce \otimes_{eAe} ed))  \\
   & =  \gamma'(xae\otimes_{eAe} eb) \otimes_k \gamma(ce \otimes_{eAe} ed)\\
 xm_2( \gamma(ae\otimes_{eAe} eb) \otimes_k \gamma(ce \otimes_{eAe} ed)) & =  x\gamma'(ae\otimes_{eAe} eb) \otimes_k \gamma(ce \otimes_{eAe} ed) \\
   & =  \gamma'(xae\otimes_{eAe} eb) \otimes_k \gamma(ce \otimes_{eAe} ed).
\end{align*}
Therefore, $m_2(x \gamma(ae\otimes_{eAe} eb)\otimes_k \gamma(ce \otimes_{eAe} ed))= xm_2(\gamma(ae\otimes_{eAe} eb)\otimes_k \gamma(ce \otimes_{eAe} ed))$. Also,
\begin{align*}
  m_2(\gamma(ae\otimes_{eAe} eb) \otimes_k \gamma(ce \otimes_{eAe} ed)y) & =  m_2(\gamma(ae\otimes_{eAe} eb) \otimes_k \gamma(ce \otimes_{eAe} ed\sigma(y)))  \\
   & =  \gamma'(ae\otimes_{eAe} eb) \otimes_k \gamma(ce \otimes_{eAe} ed\sigma(y))\\
  m_2( \gamma(ae\otimes_{eAe} eb) \otimes_k \gamma(ce \otimes_{eAe} ed))y & =  \gamma'(ae\otimes_{eAe} eb) \otimes_k \gamma(ce \otimes_{eAe} ed)y \\
   & =  \gamma'(ae\otimes_{eAe} eb) \otimes_k \gamma(ce \otimes_{eAe} ed\sigma(y)).
\end{align*}
Hence, $m_2(\gamma(ae\otimes_{eAe} eb)\otimes_k \gamma(ce \otimes_{eAe} ed)y)= m_2(\gamma(ae\otimes_{eAe} eb)\otimes_k \gamma(ce \otimes_{eAe} ed))y$. This means that $m_2$ is an $A$-bimodule morphism. \hspace{13.7cm} $\Box$ \\

Let $m$ be the following composition map
\[m : \text{D}(A)\otimes_k \text{D}(A) \stackrel{m_2}{\rightarrow} \text{D}(A)_{\sigma^{-1}}\otimes_k \text{D}(A)\stackrel{m_1}{\rightarrow} \text{D}(A),\]
where
\begin{align*}
  m : \gamma(ae\otimes_{eAe} eb) \otimes_k \gamma(ce\otimes_{eAe} ed) & \mapsto \gamma'(ae\otimes_{eAe} eb) \otimes_k \gamma(ce\otimes_{eAe} ed) \\
  & \mapsto \gamma(aebce\otimes_{eAe} ed). \nonumber
\end{align*}
Dualising $m$ yields an $A$-bimodule morphism
\[\Delta: A \rightarrow  A \otimes_k A\]
such that
\[(f\otimes g)\Delta(x)=m(g\otimes f)(x)\]
for any $f,g$ in D($A$) and $x$ in $A$.
\begin{theorem}\label{thmcomgenFrob} Let $A$ be a gendo-Frobenius algebra. Then
\[\Delta: A \rightarrow A \otimes_k A\]
is a coassociative comultiplication which is an $A$-bimodule morphism.
\end{theorem}
The proof of Theorem \ref{thmcomgenFrob} consists of the following two lemmas.
\begin{lemma}\label{lemmamultp} The map $m$ satisfies
\[m(1\otimes m)= m(m\otimes 1)\]
as $k$-morphisms from $\text{D}(A)\otimes_k \text{D}(A)\otimes_k \text{D}(A)$ to $\text{D}(A)$.
\end{lemma}
\begin{proof} The definition of $m$ above implies that
\begin{align*}
  m(1 \otimes m) (\gamma(ae\otimes eb) \otimes_k \gamma(ce\otimes ed) \otimes_k \gamma(xe\otimes ey)) & = m (\gamma(ae\otimes eb) \otimes_k \gamma(cedxe \otimes ey)) \\
   & = \gamma(aebcedxe \otimes ey)\\
  m(m \otimes 1) (\gamma(ae\otimes eb) \otimes_k \gamma(ce\otimes ed) \otimes_k \gamma(xe\otimes ey)) & = m (\gamma(aebce\otimes ed) \otimes_k \gamma(xe \otimes ey)) \\
   & = \gamma(aebcedxe \otimes ey)
\end{align*}
for any $a,b,c,d,x,y \in A$. Therefore, $m(1\otimes m)= m(m\otimes 1)$.
\end{proof}
\begin{remark} We can give an alternative approach to the proof of Lemma \ref{lemmamultp} and also to writing the comultiplication $\Delta: A \rightarrow A \otimes_k A$ by using $Ae \otimes_{eAe} eA_{\sigma}$ instead of $\text{D}(A)$ since $Ae \otimes_{eAe} eA_{\sigma} \cong \text{D}(A)$ as $A$-bimodules (see Lemma \ref{lemmagamma}).
\end{remark}
\begin{lemma}\label{lemmagendoFrobcomprop} Let $\Delta : A \rightarrow A \otimes_k A$ be as above. Then

(i) $\Delta$ is an $A$-bimodule morphism.

(ii) $(1\otimes \Delta)\Delta = (\Delta \otimes 1)\Delta.$
\end{lemma}
\begin{proof}
(i) By definition of $\Delta$, we obtain the following equalities
\begin{align*}
  (\gamma(ae\otimes eb)\otimes \gamma(ce\otimes ed))\Delta(xy) & = m(\gamma(ce\otimes ed)\otimes \gamma(ae\otimes eb))(xy) \\
   & = \gamma(cedae \otimes eb)(x1y)=\gamma(ycedae \otimes eb\sigma(x))(1)\\
  (\gamma(ae\otimes eb)\otimes \gamma(ce\otimes ed))x\Delta(y) & = \gamma (ae \otimes eb\sigma(x)) \otimes \gamma(ce \otimes ed)\Delta(y)\\
   &= \gamma(cedae \otimes eb \sigma(x))(y)= \gamma(ycedae \otimes eb\sigma(x))(1)\\
  (\gamma(ae\otimes eb)\otimes \gamma(ce\otimes ed))\Delta(x)y & = (\gamma(ae \otimes eb) \otimes \gamma(yce \otimes ed))\Delta(x) \\
   & =\gamma(ycedae \otimes eb)(x) = \gamma(ycedae \otimes eb\sigma(x))(1)
\end{align*}
for $a,b,c,d,x,y \in A$. Therefore, $\Delta(xy)= x\Delta(y)= \Delta(x)y$, that is, $\Delta$ is an $A$-bimodule morphism.

(ii) Let $\Delta(u)= \sum u_i \otimes v_i$ for $u \in A$. Then
\begin{align*}
  (\gamma(ae\otimes eb)\otimes \gamma(ce \otimes ed) \otimes \gamma(xe \otimes ey))(1\otimes \Delta)\Delta(u)
   & = \sum\gamma(ae\otimes eb)(u_i)(\gamma(ce \otimes ed) \otimes \gamma(xe \otimes ey))\Delta(v_i)\\
   & = \sum\gamma(ae\otimes eb)(u_i) \gamma(xeyce\otimes ed)(v_i)  \\
   & = \gamma(ae \otimes eb) \otimes \gamma(xeyce \otimes ed)\Delta(u)\\
   & = \gamma(xeycedae \otimes eb)(u)\\
   (\gamma(ae\otimes eb)\otimes \gamma(ce \otimes ed) \otimes \gamma(xe \otimes ey))(\Delta \otimes 1)\Delta(u)
   & = \sum \gamma(ae\otimes eb)\otimes \gamma(ce \otimes ed)\Delta(u_i) \gamma(xe\otimes ey)(v_i)\\
   & = \sum \gamma(cedae\otimes eb)(u_i)\gamma(xe \otimes ey)(v_i)  \\
   & = \gamma(cedae \otimes eb) \otimes \gamma(xe \otimes ey)\Delta(u)\\
   & = \gamma(xeycedae \otimes eb)(u)
\end{align*}
This means that $(1\otimes \Delta)\Delta = (\Delta \otimes 1)\Delta$.
\end{proof}
\begin{remark} There are further constructions possible that yield comultiplications on gendo-Frobenius algebras. However, these are lacking crucial properties such as being coassociative.
\end{remark}
\begin{proposition}\label{propgendoFrobcomprop} Let $A$ be a gendo-Frobenius algebra and $\Delta : A \rightarrow A \otimes_k A$ be as above. Then
\[\text{Im}(\Delta)=\{\sum u_i \otimes v_i \mid  \sum u_i x \otimes v_i = \sum u_i \otimes \sigma^{-1}(x) v_i, \hspace{0.2cm} \forall x \in A\}.\]
\end{proposition}
\begin{proof}
Let $\Sigma = \{\sum u_i \otimes v_i \mid  \sum u_i x \otimes v_i = \sum u_i \otimes \sigma^{-1}(x) v_i, \hspace{0.2cm} \forall x \in A\}$. Let $\Delta(u)= \sum u_i \otimes v_i$, for any $u\in A$. Then for any $f,g \in \text{D}(A)$ and $x\in A$,
\[(f\otimes g)(\sum u_i x \otimes v_i)= (xf \otimes g)\Delta(u) = m(g \otimes x f)(u)\]
\[(f\otimes g)(\sum u_i  \otimes \sigma^{-1}(x) v_i)= (f \otimes g \sigma^{-1}(x))\Delta(u) = m(g \sigma^{-1}(x) \otimes  f)(u).\]
By definition of $m$, there is an equality $ m(g \otimes_k x f) = m(g \sigma^{-1}(x) \otimes_k  f)$. Because, let $f=\gamma(ae\otimes eb)$ and $g=\gamma(ce\otimes ed)$, then
\begin{align*}
  m(g \otimes_k xf) & = m_1 m_2 (\gamma(ce\otimes ed)\otimes_k x\gamma(ae\otimes eb)) = m_1 m_2 (\gamma(ce\otimes ed)\otimes_k \gamma(xae\otimes eb))\\
  & = m_1 (\gamma'(ce\otimes ed)\otimes_k \gamma(xae\otimes eb)) = \gamma(cedxae\otimes eb)\\
  m(g\sigma^{-1}(x) \otimes_k f) & = m_1 m_2 (\gamma(ce\otimes ed)\sigma^{-1}(x) \otimes_k \gamma(ae\otimes eb)) = m_1 m_2 (\gamma(ce\otimes edx)\otimes_k \gamma(ae\otimes eb))\\
  & = m_1 (\gamma'(ce\otimes edx)\otimes_k \gamma(ae\otimes eb)) = \gamma(cedxae\otimes eb).
\end{align*}
Thus $\Delta(u)\in \Sigma$ and so Im$(\Delta)\subseteq \Sigma$.

Conversely, for each $\theta = \sum u_i \otimes v_i \in \Sigma$, there is a $k$-linear map $\text{D}(A)\rightarrow A$, denoted by $\overline{\theta}$, such that $\overline{\theta}(f)= \sum f(u_i) v_i$ for any $f \in \text{D}(A)$. Since for any $x \in A$, $\sum u_i x \otimes v_i = \sum u_i \otimes \sigma^{-1}(x) v_i$, it follows
\[\overline{\theta}(x f) = \sum (x f)(u_i)v_i = \sum f(u_i x) v_i = \sum f(u_i)\sigma^{-1}(x) v_i = \sigma^{-1}(x) \overline{\theta}(f). \]
Then $\overline{\theta}$ is a left $A$-module morphism, that is, $\overline{\theta} \in \text{Hom}_A ({}_{\sigma}\text{D}(A),A) \cong \text{Hom}_A (\text{D}(A),{}_{\sigma^{-1}}A)$. Since $\text{D}(A)_{\sigma^{-1}}\otimes_A \text{D}(A) \cong \text{D}(A)$ as $A$-bimodules, by taking the dual of this isomorphism, we obtain that $\text{Hom}_A (\text{D}(A),{}_{\sigma^{-1}}A)\cong A$ as $A$-bimodules. Therefore, $\text{Hom}_A ({}_{\sigma}\text{D}(A),A) \cong A$ as $A$-bimodules. Now, observe that the map $\xi: \Sigma \rightarrow \text{Hom}_A ({}_{\sigma}\text{D}(A),A)$ which sends $\theta$ to $\overline{\theta}$ is injective. To show that it is enough to prove Ker$\xi = \{0\}$. In fact, $\xi(\theta)=\xi(\sum u_i \otimes v_i)=\overline{\theta}=0$ means that $\overline{\theta}(f)= \sum f(u_i) v_i = 0$ for any $f \in \text{D}(A)$. So we obtain that $u_i=0$ or $v_i = 0$. Therefore, $\theta=0$. Also, since $m$ is surjective, $\Delta$ is injective. Then by using Im$\Delta \subseteq \Sigma$ and previous facts, we obtain the composition of following injective maps
\[\text{Im}(\Delta) \rightarrow \Sigma \rightarrow \text{Hom}_A ({}_{\sigma}\text{D}(A),A) \cong A \rightarrow \text{Im}(\Delta).\]
Therefore, Im$\Delta= \Sigma$.
\end{proof}
Let $A$ be a gendo-Frobenius algebra with comultiplication $\tilde{\Delta}$ which satisfies Lemma \ref{lemmagendoFrobcomprop} (i) and (ii). Suppose that
\[\text{Im}(\tilde{\Delta})=\{\sum u_i \otimes v_i \mid  \sum u_i x \otimes v_i = \sum u_i \otimes \omega^{-1}(x) v_i, \hspace{0.2cm} \forall x \in A\}\]
for an automorphism $\omega$ of $A$. $\tilde{\Delta}$ induces a map $\tilde{m} : \text{D}(A) \otimes_k \text{D}(A) \rightarrow \text{D}(A)$ such that $(f\otimes g)\tilde{\Delta}(a)= \tilde{m}(g\otimes f)(a)$ for any $f, g \in \text{D}(A)$ and $a \in A$. Then $\tilde{m}$ is an $A$-bimodule morphism and it factors through $\text{D}(A)_{\omega^{-1}} \otimes_A \text{D}(A)$. Moreover, $\tilde{m}$ induces an $A$-bimodule isomorphism $\text{D}(A)_{\omega^{-1}} \otimes_A \text{D}(A) \cong \text{D}(A)$. Indeed, by Theorem \ref{PropgFsigma}, $\omega(a)=\sigma(uau^{-1})$, where $u$ is an invertible element of $A$. Then Im$(\tilde{\Delta})= \text{Hom}_A ({}_{\omega} \text{D}(A),A)\cong A$ as $A$-bimodules.
\begin{corollary} Let $A$ be a gendo-Frobenius algebra and $\tilde{\Delta}: A\rightarrow A\otimes_k A$ be as above. Then $\text{Im}(\Delta)\cong \text{Im}(\tilde{\Delta})$ as $A$-bimodules.
\end{corollary}
\begin{example}\label{examplegenFrob2} Let $A$ be the gendo-Frobenius algebra in Example \ref{GenFrobEx2}. $A$ has a $k$-basis $\{e_1, e_2, e_3, e_4, \alpha_1,$ $\alpha_2, \alpha_3, \alpha_4 , \alpha_1\alpha_3, \alpha_2\alpha_4\}$ so $\text{D}(A)$ has the dual basis $\{e_1^*, e_2^*, e_3^*, e_4^*, \alpha_1^*, \alpha_2^*, \alpha_3^*, \alpha_4^* , (\alpha_1\alpha_3)^*, (\alpha_2\alpha_4)^*\}$. We choose $e=e_1+e_2$ since $e_1+e_2$ is a faithful and self-dual idempotent of $A$. The multiplication rule on $\text{D}(A)$ described in this section is given by

\begin{center}
    \begin{tabular}{ | l | l | l | l | l | l | l | l | l | l | l |}
    \hline
    m           & $e_1^*$ & $e_2^*$ & $e_3^*$ &$e_4^*$ & $\alpha_1^*$ & $\alpha_2^*$ & $\alpha_3^*$ & $\alpha_4^*$ & $(\alpha_1\alpha_3)^*$ & $(\alpha_2\alpha_4^*)$ \\ \hline
    $e_1^*$     & 0 & 0 & 0 & 0 & 0 & 0 & 0 & 0 & $e_1^*$ &  0      \\ \hline
    $e_2^*$     & 0 & 0 & 0 & 0 & 0 & 0 & 0 & 0 & 0 & $e_2^*$    \\ \hline
    $e_3^*$     & 0 & 0 & 0 & 0 & 0 & 0 & 0 & 0 & 0 & 0       \\ \hline
    $e_4^*$     & 0 & 0 & 0 & 0 & 0 & 0 & 0 & 0 & 0 & 0      \\ \hline
    $\alpha_1^*$    & 0 & 0 & 0 & 0 & 0 & 0 & $e_3^*$  & 0 & $\alpha_1^*$ & 0  \\ \hline
    $\alpha_2^*$    & 0 & 0 & 0 & 0 & 0 & 0 & 0 & $e_4^*$ & 0 & $\alpha_2^*$  \\ \hline
    $\alpha_3^*$    & 0 & 0 & 0 & 0 & 0 & $e_2^*$ & 0 & 0 & 0 & 0  \\ \hline
    $\alpha_4^* $   & 0 & 0 & 0 & 0 &  $e_1^*$ & 0 & 0 & 0 & 0 & 0   \\ \hline
    $(\alpha_1\alpha_3)^* $ & 0 & $e_2^*$  & 0 & 0 & 0 & 0 & $\alpha_3^*$ & 0 & $(\alpha_1\alpha_3)^*$ & 0 \\ \hline
    $(\alpha_2\alpha_4)^* $ & $e_1^*$ & 0 & 0 & 0 & 0 & 0 & 0 &  $\alpha_4^*$  & 0 & $(\alpha_2\alpha_4)^*$  \\ \hline
    \end{tabular}
\end{center}
By description of $\Delta$, we obtain that
\begin{align*}
\Delta(e_1)&= \alpha_1\alpha_3 \otimes e_1 + \alpha_1 \otimes \alpha_4 + e_1 \otimes \alpha_2\alpha_4\\
\Delta(e_2)&= \alpha_2\alpha_4 \otimes e_2 + \alpha_2 \otimes \alpha_3 + e_2 \otimes \alpha_1\alpha_3\\
\Delta(e_3)&= \alpha_3 \otimes \alpha_1\\
\Delta(e_4)&= \alpha_4 \otimes \alpha_2\\
\Delta(\alpha_1)&= \alpha_1\alpha_3 \otimes \alpha_1\\
\Delta(\alpha_2)&= \alpha_2\alpha_4 \otimes \alpha_2\\
\Delta(\alpha_3)&= \alpha_3 \otimes \alpha_1\alpha_3\\
\Delta(\alpha_4)&= \alpha_4 \otimes \alpha_2\alpha_4\\
\Delta(\alpha_1\alpha_3)&= \alpha_1\alpha_3 \otimes \alpha_1\alpha_3\\
\Delta(\alpha_2\alpha_4)&= \alpha_2\alpha_4 \otimes \alpha_2\alpha_4.
\end{align*}

Let $a \in A$. Then we can write $a=a_1 e_1 + a_2 e_2 + a_3 e_3 + a_4 e_4 +a_5 \alpha_1 + a_6 \alpha_2 + a_7 \alpha_3 + a_8 \alpha_4 + a_9 \alpha_1\alpha_3 + a_{10} \alpha_2\alpha_4$, where $a_i \in k$ for $1 \leq i \leq 10$. The linearity of $\Delta$ gives that
\begin{align*}
\Delta(a) & =a_1 \Delta(e_1) + a_2 \Delta(e_2) + a_3 \Delta(e_3) + a_4 \Delta(e_4) \\
  & + a_5 \Delta(\alpha_1) + a_6 \Delta(\alpha_2) + a_7 \Delta(\alpha_3) + a_8 \Delta(\alpha_4) \\
  & + a_9 \Delta(\alpha_1\alpha_3) + a_{10} \Delta(\alpha_2\alpha_4).
\end{align*}
\end{example}

Observe that the algebra $A$ in Example \ref{examplegenFrob2} is not Frobenius. Therefore, it is natural to ask whether the algebra $A$ has a counit compatible with $\Delta$ or not. Indeed, $(A, \Delta)$ does not have a counit. Proposition \ref{propgenFrobFrobcounit} will explain why $(A, \Delta)$ does not have a counit and it will describe a general situation.
\begin{remark}\label{remarktheta}
Let us consider the following $A$-bimodule isomorphism
\begin{align*}
\text{Hom}_A (\text{D}(A),A_{\sigma}) & \cong \text{Hom}_A (\text{D}(A)_{\sigma^{-1}},A)\\
  & \cong \text{Hom}_A (Ae \otimes_{eAe} eA,A) \\
  & \cong \text{Hom}_{eAe} (eA,eA) \\
  & \cong A
\end{align*}
where the second isomorphism is Hom$_A(\gamma', A)$. Let
$\Theta : \text{D}(A) \rightarrow A_{\sigma}$
be the inverse image of $1\in A$ under the above isomorphism. Then $(\Theta \circ \gamma)(ae\otimes eb)=aeb$ for $a,b \in A$. Actually, $\Theta$ is an $A$-bimodule morphism with $e\Theta=\tau^{-1}$.

The following observation will be used to prove Proposition \ref{propFrobeqv}.\\

From $\tau : eA_{\sigma}\cong \text{D}(Ae)$, we get $\tau': eA \cong \text{D}(Ae)_{\sigma^{-1}}$. Let us now consider the following $A$-bimodule isomorphism
\begin{align*}
 \text{Hom}_A (\text{D}(A)_{\sigma^{-1}},A) & \cong \text{Hom}_A (Ae \otimes_{eAe} eA,A) \\
  & \cong \text{Hom}_{eAe} (eA,eA) \\
  & \cong A
\end{align*}
where the first isomorphism is Hom$_A(\gamma', A)$. Let
$\Theta' : \text{D}(A)_{\sigma^{-1}} \rightarrow A$
be the inverse image of $1\in A$ under the above isomorphism. Then $(\Theta' \circ \gamma')(ae\otimes eb)=aeb$ for $a,b \in A$. Actually, $\Theta'$ is an $A$-bimodule morphism with $e\Theta'=\tau'^{-1}$.
\end{remark}
\begin{lemma}\label{lemmatheta} Let $A$ be a gendo-Frobenius algebra and $m : \text{D}(A) \otimes_k \text{D}(A) \rightarrow \text{D}(A)$ as before. Then
\[\Theta(m(f\otimes g))= \Theta(f)\Theta(g)\]
for any $f,g \in \text{D}(A)$.
\end{lemma}
\begin{proof} Let $f=\gamma(ae\otimes eb)$ and $g=\gamma(ce\otimes ed)$. Then observe that
\[(\Theta\circ m )(\gamma(ae\otimes eb)\otimes \gamma(ce\otimes ed))= \Theta(\gamma(aebce\otimes ed))=aebced=(aeb)(ced)\]
\[\Theta(\gamma(ae\otimes eb))\Theta(\gamma(ce\otimes ed))=(aeb)(ced).\]
\end{proof}
We now compare the comultiplication $\Delta: A \rightarrow A\otimes_k A$ constructed in this article and the comultiplication $\alpha: A \rightarrow A\otimes_k A$ given by Abrams (Theorem \ref{thm1}) by assuming that $A$ is Frobenius.\\

We keep the notations introduced in this section. If $A$ is Frobenius, we choose $e=1_A$ and have the $A$-bimodule isomorphism $\tau : A_{\sigma} \cong \text{D}(A)$ such that $\sigma$ is a Nakayama automorphism of $A$.
\begin{proposition}\label{propFrobeqv} Let $A$ be a Frobenius algebra with the left $A$-module isomorphism $\lambda_L : A \cong \text{D}(A)$ which defines an isomorphism $\lambda_L : A_{\sigma}\cong \text{D}(A)$ of $A$-bimodules, where $\sigma$ is a Nakayama automorphism of $A$. Suppose that $\lambda_L= \tau$. Then $\alpha$ is equal to $\Delta$.
\end{proposition}
\begin{proof} Let $A$ be Frobenius and $\tau : A_{\sigma} \cong \text{D}(A)$ be the $A$-bimodule isomorphism. We can consider $\tau$ as $\tau' : A \rightarrow \text{D}(A)_{\sigma^{-1}}$ such that $\tau(a)=\tau'(a)$ for any $a\in A$. Therefore, $\lambda_L(a)=\tau'(a)$ for any $a\in A$. Moreover, there is an $A$-bimodule isomorphism $\gamma : A \otimes_A A_{\sigma} \cong \text{D}(A)$ by Lemma \ref{lemmagamma} and so $\gamma' : A \otimes_A A \cong \text{D}(A)_{\sigma^{-1}}$. By Remark \ref{remarktheta}, we have an $A$-bimodule isomorphism $\Theta' : \text{D}(A)_{\sigma^{-1}} \rightarrow A $ with $\Theta' = \tau'^{-1}$. By following the same remark, we write $\tau'^{-1}(\gamma'(x \otimes y))= xy$ for any $x,y \in A$.

Since the Frobenius form $\varepsilon$ of $A$ is equal to $\lambda_L(1_A)$, all elements of $\text{D}(A)$ are of the form $a \cdot \varepsilon$ for any $a \in A$. The left $A$-module isomorphism $\lambda_L : A \cong \text{D}(A)$ allows us to define a multiplication $\varphi_L:= \lambda_L \circ \mu \circ (\lambda_L^{-1}\otimes \lambda_L^{-1})$ such that $\varphi_L (a\cdot \varepsilon \otimes b \cdot \varepsilon) = (b \cdot \varepsilon \otimes a\cdot \varepsilon) \circ \alpha_R = ab \cdot \varepsilon$.

Let $\vartheta : A \otimes_A A \cong A$ be the $A$-bimodule isomorphism such that $\vartheta (a \otimes_A b)=ab$ and $\mu' : A \otimes_k A \rightarrow A\otimes_A A$ be the map such that $\mu'(a\otimes_k b)= a\otimes_A b$ for any $a,b\in A$. Suppose that $\lambda_L' := \lambda_L \circ \vartheta$ and $\varphi_L' := \lambda_L' \circ \mu' \circ (\lambda_L^{-1} \otimes \lambda_L^{-1})$. Then observe the following
$$\xymatrixrowsep{0.1in}
\xymatrixcolsep{3pc}\xymatrix{
& \varphi_L':  \text{D}(A) \otimes_k \text{D}(A) \ar[r]^-{\lambda_L^{-1}\otimes \lambda_L^{-1}} & A \otimes_k A  \ar[r]^-{\mu'} & A \otimes_A A  \ar[r]^-{\lambda_L'} & \text{D}(A) \\
 & a\cdot \varepsilon \otimes_k b\cdot \varepsilon   \ar@{|->}[r] & a \otimes_k b   \ar@{|->}[r] & a \otimes_A b   \ar@{|->}[r] &  ab\cdot \varepsilon
   }$$
Therefore, $\varphi_L = \varphi_L'$.

Observe that there are isomophisms of left $A$-modules $\tau' \otimes_k \lambda_L : A \otimes_k A \cong \text{D}(A)_{\sigma^{-1}}\otimes_k \text{D}(A)$ and $\tau' \otimes_A \lambda_L : A \otimes_A A \cong \text{D}(A)_{\sigma^{-1}}\otimes_A \text{D}(A)$. We now observe the following diagram
$$\begin{xy}
\xymatrixrowsep{0.2in}
\xymatrixcolsep{0.2in}
  \xymatrix{
  & \text{D}(A)\otimes_k \text{D}(A) \ar[r]^-{\lambda_L^{-1} \otimes_k \lambda_L^{-1}} \ar[dr]_{m_2} & A\otimes_k A \ar[r]^{\mu'} \ar[d]^{\tau' \otimes_k \lambda_L} & A\otimes_A A \ar[r]^{\lambda_L'} \ar[d]_{\tau' \otimes_A \lambda_L} & \text{D}(A)\\
  & &\text{D}(A)_{\sigma^{-1}}\otimes_k \text{D}(A)\ar[r]^{\phi} & \text{D}(A)_{\sigma^{-1}}\otimes_A \text{D}(A) \ar[ur]_{\epsilon} &
    }
\end{xy}$$

Since $\gamma: A\otimes_A A_{\sigma} \cong \text{D}(A)$ as $A$-bimodules and $\varepsilon \in \text{D}(A)$, we can write $\varepsilon = \gamma(x \otimes_A y)$ for suitable $x,y \in A$. Since $\text{D}(A)=\text{D}(A)_{\sigma^{-1}}$ as $k$-vector spaces, we can consider $\varepsilon$ as $\varepsilon=\gamma'(x\otimes_A y)$ when we need to use it. Then any $a \cdot \varepsilon$ of $\text{D}(A)$ can be written as $a \cdot \varepsilon = \gamma(ax \otimes_A y)$ and any $a \cdot \varepsilon$ of $\text{D}(A)_{\sigma^{-1}}$ can be written as $a \cdot \varepsilon = \gamma'(ax \otimes_A y)$. Therefore, $ \lambda_L^{-1}(\gamma(ax \otimes_A y))= \lambda_L^{-1}(a\cdot \varepsilon)= a$. Then $\tau'^{-1} (\gamma'(ax \otimes_A y))=axy =a $ by definition of $\tau'^{-1}$ given above. Since $A$ is faithful $A$-module, $xy=1$. Moreover, $(\tau' \otimes_k \lambda_L) (a \otimes_k b) = \gamma'(ax \otimes_A y) \otimes_k \gamma(bx \otimes_A y)$  and $(\tau' \otimes_A \lambda_L) (a \otimes_A b) = \gamma'(ax \otimes_A y) \otimes_A \gamma(bx \otimes_A y)$. Recall that $m=\epsilon \circ \phi \circ m_2$.

Then by using the above information, first observe that
\begin{align*}
  (\tau' \otimes_k \lambda_L) \circ (\lambda_L^{-1} \otimes_k \lambda_L^{-1})(a \cdot \varepsilon \otimes_k b \cdot \varepsilon) & = (\tau' \otimes_k \lambda_L)(a \otimes_k b) \\
   & = \gamma'(ax \otimes_A y) \otimes_k \gamma(bx \otimes_A y) \\
  m_2 (a \cdot \varepsilon \otimes_k b \cdot \varepsilon) & = m_2 (\gamma(ax \otimes_A y) \otimes_k \gamma(bx \otimes_A y)) \\
   & = \gamma'(ax \otimes_A y) \otimes_k \gamma(bx \otimes_A y).
\end{align*}
It means that left part of the above diagram is commutative.

Also, we see that
\begin{align*}
 (\tau' \otimes_A \lambda_L) \circ \mu' (a \otimes_k b) & = (\tau' \otimes_A \lambda_L) (a \otimes_A b)   \\
   & = \gamma'(ax \otimes_A y) \otimes_A \gamma(bx \otimes_A y)\\
  \phi \circ (\tau' \otimes_k \lambda_L)(a\otimes_k b) & =  \phi (\gamma'(ax \otimes_A y) \otimes_k \gamma(bx \otimes_A y)) \\
   & = \gamma'(ax \otimes_A y) \otimes_A \gamma(bx \otimes_A y).
\end{align*}
Hence, middle part of the diagram is commutative.

Moreover, we have
\begin{align*}
  \epsilon \circ (\tau' \otimes_A \lambda_L)(a \otimes_A b) & = \epsilon (\gamma'(ax \otimes_A y) \otimes_A \gamma(bx \otimes_A y)) \\
  & = \gamma(axybx \otimes_A y) \\
  & = \gamma(abx \otimes_A y) \\
  & = ab \cdot \varepsilon \\
  \lambda_L'(a\otimes_A b) & = ab \cdot \varepsilon.
\end{align*}
Therefore, right part of the diagram is commutative. This means that $\varphi_L'=m$ and so $\varphi_L=m$. Then dualising gives that $\alpha_R=\Delta$.

There is also a comultiplication $\alpha_L$ which is a map of left $A$-modules and in \cite{Abrams}, Abrams proved that $\alpha_L=\alpha_R$ and defined $\alpha:=\alpha_L=\alpha_R$. Hence, we obtain that $\alpha=\Delta$.
\end{proof}
The above proposition shows that the comultiplication constructed by Abrams (Theorem \ref{thm1}) and the comultiplication constructed in this article are equal when the algebra $A$ is Frobenius.
\begin{proposition}\label{propgenFrobFrobcounit} Let $A$ be a gendo-Frobenius algebra with the comultiplication $\Delta: A \rightarrow A \otimes_k A$. Then $(A, \Delta)$ has a counit if and only if $A$ is Frobenius.
\end{proposition}
\begin{proof} Let $\delta \in \text{D}(A)$ be a counit of $(A,\Delta)$. Then $m(\delta \otimes f)(a)= (f\otimes \delta)\Delta(a)= f(1\otimes \delta)\Delta(a)=f(a)$, and similarly $m(f\otimes \delta)(a)=(\delta\otimes f)\Delta(a)=f(a)$ for any $a\in A$. Therefore, $\delta$ is a unit of $(\text{D}(A),m)$. Now, let $u$ be the image of $\delta$ under $\Theta: \text{D}(A)\rightarrow A_{\sigma}$. Then $\Theta m(\delta \otimes \gamma (ae\otimes eb))= \Theta (\gamma(ae\otimes eb))$. So, we obtain that $uaeb=aeb$ for any $a,b \in A$ by Lemma \ref{lemmatheta}. Hence, we obtain that $u=1$ since $AeA$ is a faithful left $A$-module. As a result, $\Theta$ is surjective as an $A$-bimodule morphism and thus an isomorphism by comparing dimensions. So $A$ is Frobenius. In fact, $\sigma$ is a Nakayama automorphism of $A$.

Conversely, let $A$ be Frobenius. Then, by Theorem \ref{thm1} and Proposition \ref{propFrobeqv}, $(A,\Delta)$ has a counit.
\end{proof}
In particular, the case $A$ is Frobenius, which is proved by Abrams \cite{Abrams}, is obtained as a special case of Theorem \ref{thmcomgenFrob} and Proposition \ref{propgenFrobFrobcounit}. In addition, Proposition \ref{propgendoFrobcomprop} is specialised to Frobenius algebras.
\begin{corollary} Let $A$ be a Frobenius algebra. Then it has a coassociative counital comultiplication $\Delta: A \rightarrow A \otimes_k A$ which is an $A$-bimodule morphism such that
\[\text{Im}(\Delta)=\{\sum u_i \otimes v_i \mid  \sum u_i x \otimes v_i = \sum u_i \otimes \nu_A^{-1}(x) v_i, \hspace{0.2cm} \forall x \in A\},\]
where $\nu_A$ is a Nakayama automorphism of $A$.
\end{corollary}
Moreover, the case $A$ is gendo-symmetric, which is proved by Fang and Koenig (Theorem 2.4 \& Lemma 2.6, \cite{Koenig1}), is obtained as a special case of Theorem \ref{thmcomgenFrob} and Proposition \ref{propgendoFrobcomprop}.
\begin{corollary} Let $A$ be a gendo-symmetric algebra. Then it has a coassociative comultiplication $\Delta: A \rightarrow A \otimes_k A$ which is an $A$-bimodule morphism such that
\[\text{Im}(\Delta)=\{\sum u_i \otimes v_i \mid  \sum u_i x \otimes v_i = \sum u_i \otimes x v_i, \hspace{0.2cm} \forall x \in A\}.\]
\end{corollary}
\begin{remark} If we assume that the finite dimensional algebra $A$ is gendo-symmetric, we can choose $\sigma$ as identity automorphism. Therefore, the comultiplication given in this article and the comultiplication given by Fang and Koenig in \cite{Koenig1} are equal for gendo-symmetric algebras.
\end{remark}
The last two results show that the comultiplication $\Delta: A \rightarrow A\otimes_k A$ constructed in this article is a common comultiplication for Frobenius algebras and gendo-symmetric algebras.\\

\textbf{Acknowledgements.} The author would like to thank Steffen Koenig for helpful comments and proof reading. The results in this article are a part of author's doctoral thesis \cite{Yirtici}, which was financially supported by DFG.

\end{document}